\documentclass[12pt,reqno]{amsart}
\usepackage{amsmath,amsthm,amsfonts,amssymb,times}
\usepackage{verbatim}
\usepackage{url}
\setbox0=\hbox{$+$}
\newdimen\plusheight
\plusheight=\ht0
\def\+{\;\lower\plusheight\hbox{$+$}\;}

\setbox0=\hbox{$-$}
\newdimen\minusheight
\minusheight=\ht0
\def\-{\;\lower\minusheight\hbox{$-$}\;}

\setbox0=\hbox{$\cdots$}
\newdimen\cdotsheight
\cdotsheight=\plusheight
\def\cds{\lower\cdotsheight\hbox{$\cdots$}}

\makeatletter
\def\leqalignno#1{\displ@y \tabskip\z@ plus\@ne fil
  \halign to\displaywidth{\hfil$\@lign\displaystyle{##}$\tabskip\z@skip
    &$\@lign\displaystyle{{}##}$\hfil\tabskip\z@ plus\@ne fil
    &\kern-\displaywidth\rlap{$\@lign\hbox{\rm##}$}\tabskip\displaywidth\crcr
    #1\crcr}}
\makeatother
\let\dotlessi=\i

\newcommand{\eb}{\begin{equation}}
\newcommand{\ee}{\end{equation}}

\newcommand{\df}{\dfrac}
\newcommand{\tf}{\tfrac}

\renewcommand{\d}{{\delta}}

\renewcommand{\(}{\left\(}
\renewcommand{\)}{\right\)}
\renewcommand{\[}{\left\[}
\renewcommand{\]}{\right\]}
\renewcommand{\i}{\infty}
\renewcommand{\pmod}[1]{\,(\textup{mod}\,#1)}
\numberwithin{equation}{section}
 \theoremstyle{plain}
\newtheorem{theorem}{Theorem}[section]

\newtheorem{entry}[theorem]{Entry}

\begin{document}
\title{On a Theorem of A.~I.~Popov on Sums of Squares}
\author{Bruce C.~Berndt, Atul Dixit, Sun Kim, and Alexandru Zaharescu}
\address{Department of Mathematics, University of Illinois, 1409 West Green
Street, Urbana, IL 61801, USA} \email{berndt@illinois.edu}
\address{Department of Mathematics, Indian Institute of Technology, Gandhinagar, Palaj, Gandhinagar 382355, Gujarat, India}\email{adixit@iitgn.ac.in}
\address{Department of Mathematics, University of Illinois, 1409 West Green
Street, Urbana, IL 61801, USA} \email{sunkim2@illinois.edu}
\address{Department of Mathematics, University of Illinois, 1409 West Green
Street, Urbana, IL 61801, USA and \newline
		Simion Stoilow Institute of Mathematics of the Romanian
		Academy, P.O. Box 1--764, RO--014700 Bucharest, Romania.} \email{zaharesc@illinois.edu}

\begin{abstract}
Let $r_k(n)$ denote the number of representations of the positive integer $n$ as the sum of $k$ squares.  In 1934, the Russian mathematician A.~I.~Popov  stated, but did not rigorously prove, a beautiful series transformation involving $r_k(n)$ and certain Bessel functions.  We provide a proof of this identity for the first time, as well as for another identity, which can be regarded as both an analogue of Popov's identity and an identity involving $r_2(n)$ from Ramanujan's lost notebook.
\end{abstract}



\maketitle

\section{Introduction} In an obscure paper \cite{popov} written in Russian, A.~I.~Popov offers the following beautiful identity involving $r_k(n)$, the number of representations of the positive integer $n$ as a sum of $k$ squares.  Let $J_{\nu}(z)$ denote the ordinary Bessel function of order $\nu$ and let $I_{\nu}(z)$ denote the Bessel function of imaginary argument of order $\nu$, usually so denoted.  If
$\text{Re }t>0$, then
\begin{align}\label{88}
&\df{\pi^{k/2-1}z^{k/4-1/2}}{\Gamma(\tf12k)}+\sum_{n=1}^{\infty}r_{k}(n)\df{J_{k/2-1}(2\pi\sqrt{nz})}{n^{k/4-1/2}}e^{-\pi nt}\notag\\
&=\df{e^{-\pi z/t}}{t}\left\{\df{\pi^{k/2-1}z^{k/4-1/2}}{t^{k/2-1}\Gamma(\tf12k)}
+\sum_{n=1}^{\infty}r_{k}(n)\df{I_{k/2-1}\left(\df{2\pi\sqrt{nz}}{t}\right)}{
  n^{k/4-1/2}}e^{-\pi n/t}\right\}.
\end{align}
  Why is \eqref{88} a fascinating identity?  In the remainder of this paragraph, we attempt to convince readers to share the authors' fascination.  Recall \cite[p.~18]{cn} that if
$$ \zeta_k(s):=\sum_{n=1}^{\infty}\df{r_k(n)}{n^s}, \qquad k\geq 2, \quad \text{Re }s>\tf12k,$$
then $\zeta_k(s)$ can be analytically continued to the entire complex plane, except for a simple pole at $s=\tf12 k$ with residue $\pi^{k/2}/\Gamma(\tf12k)$,  and $\zeta_k(s)$ satisfies the functional equation
\begin{equation}\label{funcequa1}
\pi^{-s}\Gamma(s)\zeta_k(s)=\pi^{-(k/2-s)}\Gamma(k/2-s)\zeta_k(k/2-s).
\end{equation}
Second, the associated theta transformation formula is given by \cite[p.~19]{cn}
\begin{equation}\label{thetatrans}
\sum_{n=0}^{\infty}r_k(n)e^{-\pi ny}=y^{-k/2}\sum_{n=0}^{\infty}r_k(n)e^{-\pi n/y}, \qquad \text{Re }y>0.
\end{equation}
Thirdly, if $x>0$ and $q>\tf12(k-1)$, then \cite[p.~19]{cn}
{\allowdisplaybreaks\begin{align}\label{bessel1}
\df{1}{\Gamma(q+1)}&{\sum_{0\leq n\leq x}}^{\prime}r_k(n)(x-n)^{q}\nonumber\\
&\quad\quad=\df{\pi^{k/2}x^{k/2+q}}{\Gamma(q+1+k/2)}+\left(\df{1}{\pi}\right)^{q}
\sum_{n=1}^{\infty}r_k(n)\left(\df{x}{n}\right)^{k/4+q/2}J_{k/2+q}(2\pi\sqrt{nx}),
\end{align}}%
where the series on the right-hand side converges absolutely.  As indicated in Theorem \ref{cntheorem} below, we can extend the validity of \eqref{bessel1} to $q>\tf12(k-3)$.  The prime $\prime$ on the summation sign on the left-hand side indicates that if $q=0$ and $x$ is an integer $N$, then we count only $\tf12r_k(N)$.  K.~Chandrasekharan and R.~Narasimhan \cite[Theorem I, p.~10]{cn} prove a general theorem in which they show that the functional equation, theta transformation, and Bessel series identity are equivalent, i.e., if any one of these three identities  holds, then the other two also hold.  Thus, in the special case above, the functional equation \eqref{funcequa1}, theta relation \eqref{thetatrans}, and Bessel series identity \eqref{bessel1} are equivalent, i.e., if any one of \eqref{funcequa1}--\eqref{bessel1} holds, the other two are also valid.

Let us now examine \eqref{88}.  The powers of $n$ in the denominators on both sides are remindful of \eqref{funcequa1}.  The exponentials $e^{-\pi nt}$ on the left-hand side and $e^{-\pi n/t}$ on the right-hand side harken back to the theta transformation formula \eqref{thetatrans}.  Lastly, the appearance of Bessel functions in \eqref{88} reminds us of the identity \eqref{bessel1}.  Hence, the identity \eqref{88} incorporates features of all three equivalent identities, \eqref{funcequa1}, \eqref{thetatrans}, and \eqref{bessel1}!

Popov only very briefly sketches his proof of \eqref{88}, and it appears that his proof is deficient.  In particular, he uses the identity \eqref{bessel1} under the assumption that it is valid for $q>-1$, which contrasts with the requirement of Chandrasekharan and Narasimhan \cite[p.~14, Theorem III]{cn} $q>\tf12(k-3)$, which appears to be best possible.  In particular, if $q=0$, then the case $k=2$ is the only instance when \eqref{bessel1} is valid.  Popov employs the case $q=0$ of \eqref{bessel1} to derive a version of the Vorono\"{\dotlessi} summation formula, which he in turn utilizes to establish \eqref{88}. However, Popov does not offer any hypotheses for the generic function $f$ appearing in the Vorono\"{\dotlessi} summation formula.  In most formulations of the Vorono\"{\dotlessi} summation formula, the hypotheses on $f$ are too severe to permit the application that is needed to prove \eqref{88} for integers $k\geq2$.  See, for example, the paper \cite{V} providing references to several versions of the Vorono\"{\dotlessi} summation formula.   However, N.~S.~Koshliakov \cite{kosh1}, \cite{kosh2}, with a very ingenious argument, establishes a version of the Vorono\"{\dotlessi} summation formula that suffices for our purposes of proving \eqref{88} for all positive integers
$k>1$ without using (1.4), which as mentioned before, is not valid for $k>2$
when $q=0.$  (We remark that the papers \cite{kosh1} and \cite{kosh2} are identical.)  More precisely, we actually need an extension of Koshliakov's extension of the Vorono\"{\dotlessi} summation formula, but  our extension can be proved along exactly the same lines as those given by Koshliakov, and so we forego our proof.  Thus, a primary purpose of this paper is to provide not only a rigorous proof of Popov's formula \eqref{88} but to show that it can also be extended to a more general setting in which the arithmetical function $r_k(n)$ is replaced by a more general arithmetical function with its associated Dirichlet series satisfying a general functional equation, which we make precise.

Before describing our general setting, because most readers will not be familiar with his name or work, we provide a brief biography of Popov.  Alexander Ivanovich Popov  (1899--1973) was born in the Pskov region in northwest Russia.   He graduated from Leningrad University and taught at the Leningrad Polytechnic Institute.   During the period 1930--1945 he published 13 papers in mathematics.
He then turned to Finno-Ugric Linguistics (a group of languages in northeast Europe including the Finnish, Estonian, and Hungarian languages) and wrote a doctoral thesis on toponymics (the study of the origins of place names in countries).  He became perhaps the world's leading expert in Finno-Ugric Linguistics and is so regarded to this day, with his contributions to mathematics largely forgotten.  Nonetheless, after his aforementioned doctorate, he continued to teach mathematics and was the chair of the Department of Mathematical Logic and Philosophy at the  Machine-Building Institute at the Leningrad State University. During this period, he published a textbook `Introduction to Mathematical Logic', which not only contains theory but also practical applications of mathematical logic as the basis for design and operation of various ``smart'' and ``thinking'' machines \cite{gazeta}. For a biography of Popov emphasizing his contributions to linguistics, the reader is referred to \cite{bereczki}.

We close our introduction with descriptions of other theorems proved in this paper.  In his lost notebook \cite[p.~335]{lnb}, Ramanujan offered a beautiful theorem which can be considered as a two-variable analogue of \eqref{bessel1} with $k=2$ and $q=0$.
 To state Ramanujan's claim, we first need to
define
\begin{equation}\label{b.F}
F(x) = \begin{cases} [x], \qquad &\text{if $x $ is not an
integer}, \\
x-\tfrac{1}{2}, \qquad &\text{if $x$ is an integer}.
\end{cases}
\end{equation}
We now offer Ramanujan's beautiful identity.

\begin{entry}[p.~335]\label{b.besselseries} Let $F(x)$ be defined by \eqref{b.F}, and recall that $J_1(z)$ denotes the ordinary Bessel function of order $1$.
If $0 < \theta < 1$ and $x>0$, then
\begin{align}\label{b.b1.1}
\sum_{n=1}^{\infty}&F\left(\frac{x}{n}\right)\sin(2\pi n\theta)= \pi
x\left(\dfrac{1}{2}-\theta\right) -\df{1}{4}\cot(\pi\theta)\\
&+\frac{1}{2}\sqrt{x}\sum_{m=1}^{\infty}\sum_{n=0}^{\infty}
\left\{\df{J_1\left(4\pi\sqrt{m(n+\theta)x}\right)}{\sqrt{m(n+\theta)}}
-\df{J_1\left(4\pi\sqrt{m(n+1-\theta)x}\right)}{\sqrt{m(n+1-\theta)}}\right\}.\notag
\end{align}
\end{entry}

In \cite{bessel1}, the first and fourth authors proved \eqref{b.b1.1} with the order of summation of the double Bessel series on the right-hand side reversed.  Then in \cite{bessie3}, the first, third, and fourth authors established \eqref{b.b1.1} with the order of the double series indicated by Ramanujan.  Theorem \ref{pop} provides an analogue of Ramanujan's Entry \ref{b.besselseries} but in the spirit of Popov's identity \eqref{88}.

\section{Summation Formulas}\label{sumformulas}
We  use the notation and theorems from \cite{cn} and \cite{V}.  In the sequel, $\sigma=\text{Re }s$. Suppose that
\begin{equation*}
\varphi(s):=\sum_{n=1}^{\infty}a(n)\lambda_n^{-s} \quad\text{and}\quad \psi(s):=\sum_{n=1}^{\infty}b(n)\mu_n^{-s},
\end{equation*}
where $0<\lambda_1<\lambda_2<\cdots <\lambda_n\to\i$ and $0<\mu_1<\mu_2<\cdots<\mu_n\to\i$, and where the abscissae of absolute convergence are, respectively, $\sigma_a$ and $\sigma_a^*$.  We say that $\varphi(s)$ and $\psi(s)$ satisfy a functional equation of the type
\begin{equation}\label{funcequa}
\Gamma(s)\varphi(s)=\Gamma(r-s)\psi(r-s),
\end{equation}
for some $r>0$, if there exists a meromorphic function $\chi$ with the following properties:
\begin{align*}
\text{(i)}& \quad\chi(s)= \Gamma(s)\varphi(s), \quad \sigma>\sigma_a,\qquad \chi(s)= \Gamma(r-s)\psi(r-s),\quad \sigma<r-\sigma_a^*;\\
\text{(ii)}& \quad\lim_{|\text{Im }s|\to\infty}\chi(s)=0, \text{uniformly in every interval} -\infty<\sigma_1\leq\sigma\leq\sigma_2<\infty;\\
\text{(iii)}& \quad\text{the poles of $\chi$ are confined to a compact set}.
\end{align*}
  We could also consider further functional equations in which $\Gamma(s)$ is replaced by other products of gamma functions, e.g., $\Gamma^2(\tf12s)$.

   For real $q$ and $x>0$, first set
\begin{equation}\label{A}
A_q(x):=\df{1}{\Gamma(q+1)}{\sum_{\lambda_n\leq x}}^{\prime}a(n)(x-\lambda_n)^q,
\end{equation}
where the prime $\prime$ on the summation sign indicates that if $q=0$ and $x=\lambda_N$ for some positive integer $N$, then we count only $\tf12a(N)$ in the sum.    Second, define, for $q$ and $x$ as above,
\begin{equation}\label{Q}
Q_q(x):=\df{1}{2\pi i}\int_{C_q}\df{\Gamma(s)\varphi(s)}{\Gamma(s+q+1)}x^{s+q}ds,
\end{equation}
where $C_q$ is a positively oriented closed curve (or curves) with all of the integrand's poles on the interior of $C_q$.  We assume that all of the poles of the integrand lie within a compact set in the $s$-plane.  Thirdly, for the same values of $q$ and $x$ as above, define
\begin{equation}\label{D}
D_q(x):=\sum_{n=1}^{\infty}\df{b(n)}{\mu_n^{r+q}}\mathcal{I}_q(\mu_nx),
\end{equation}
where
\begin{equation}\label{bessel}
\mathcal{I}_q(x):=x^{(r+q)/2}J_{r+q}(2\sqrt{x}),
\end{equation}
where $J_{\nu}(x)$ is the ordinary Bessel function of order $\nu$.  (If we would replace $\Gamma(s)$ in the functional equation \eqref{funcequa} by another suitable product of gamma functions, then the Bessel function $J_{\nu}$ would be replaced by other appropriate Bessel functions.)  Then, for $q>2\sigma_a^*-r-\tf12$ \cite[p.~6]{cn},
\begin{equation}\label{mainbessel}
A_q(x)=Q_q(x)+D_q(x).
\end{equation}
The condition $q>2\sigma_a^*-r-\tf12$ is too weak for many applications.  In particular, in most instances it is not satisfied for $q=0$.  We thus need the following strong theorem of Chandrasekharan and Narasimhan \cite[p.~14, Theorem III]{cn}.

\begin{theorem}\label{cntheorem} Suppose that, for $\beta>\sigma_a^*$,
\begin{equation*}
\sum_{n=1}^{\infty}\df{|b(n)|}{\mu_n^{\beta}}<\infty,
\end{equation*}
\begin{equation*}
\sup_{0\leq h\leq1}\left|\sum_{m^2<\mu_n\leq(m+h)^2}\df{b(n)}{\mu_n^{\beta-\frac12}}\right|=o(1),
\end{equation*}
as $m\to\infty$, and \eqref{mainbessel} holds for some $q>0$.  Then the series of Bessel functions on the right-hand side of \eqref{mainbessel} converges for $q\geq2\beta-r-\tf32$ uniformly in any interval in $x>0$ in which the function on the left-hand side of \eqref{mainbessel} is continuous, and boundedly in any interval $0<x_1\leq x\leq x_2<\infty$ if $q=0$.
\end{theorem}

We next state a version of the Vorono\"{\dotlessi} summation formula.  We refer to Theorem 1 of \cite[p.~142]{V}, where the  Vorono\"{\dotlessi} summation formula is stated for an interval $(a,x)$, where $0<a<x$.

\begin{theorem} Let $f\in C^{(1)}[a,x]$, where $0<a<\lambda_1<x$.  Assume that \eqref{mainbessel} is valid for $q=0$.  Suppose that all of the poles of $\varphi(s)$ lie in the half-plane $\sigma>0$. Then, for $x>0$,
\begin{equation}\label{voronoi}
{\sum_{\lambda_n\leq x}}^{\prime}a(n)f(\lambda_n)=\int_{a}^x{Q}^{\prime}_0(t)f(t)dt +\sum_{n=1}^{\infty}\df{b(n)}{\mu_n^{r-1}}
\int_{a}^x\mathcal{I}_{-1}(\mu_nt)f(t)dt.
\end{equation}
\end{theorem}
\textbf{Remark.} For a function $f\in C^{r}[0,\infty)$, we can perform integration by parts $r$ times for the integral inside the sum on the right of \eqref{voronoi}, thereby increasing the power of $\mu_n$ in the bottom. This will then allow the series with the resulting integral in the summand to converge, leaving only the question about the convergence of the infinite series of terms arising from the boundary values after integrating by parts. Additional restrictions may be needed on $f$ for the latter to converge.\\

In Theorems 2 and 3 of \cite[p.~142]{V}, Theorem 1 is extended to include the case $a=0$.  However, the claim that Theorem 1 can be extended to the case $a=0$ is incorrect, because of a faulty argument by the author of \cite{V}.  The mistake arises when attempting to extend the last displayed equality in the proof of Theorem 1 to the case $a=0$.  Let us consider correcting Theorem 3, where it is assumed that all of the poles of $\varphi(s)$ lie in the half-plane $\sigma>0$.  We note that
\begin{equation}\label{Q0}
Q_{0}(a)=\df{1}{2\pi i}\int_{C_0}\df{\varphi(s)}{s}a^sds=\varphi(0)a^0+\cdots,
\end{equation}
where the terms represented by $+\cdots$ contain powers of $a$ with positive real parts, because of our assumption about the poles of $\varphi(s)$.  Hence, we do not obtain a contribution to the aforementioned equality equal to 0, as we previously claimed, but instead our contribution is
$$\lim_{a\to0}Q_{0}(a)f(a)=\lim_{a\to0}\varphi(0)f(a).$$
We now state a corrected version of Theorem 3 of \cite{V}.

\begin{theorem}\label{v} Let $f\in C^{(1)}[0,\infty)$.  Assume that \eqref{mainbessel} is valid for $q=0$.  Suppose that all of the poles of $\varphi(s)$ lie in the half-plane $\sigma>0$. Then, for $x>0$,
\begin{equation}\label{voronoib}
{\sum_{\lambda_n\leq x}}^{\prime}a(n)f(\lambda_n)=\lim_{a\to0}\varphi(0)f(a)+\int_0^x{Q}^{\prime}_0(t)f(t)dt +\sum_{n=1}^{\infty}\df{b(n)}{\mu_n^{r-1}}
\int_0^x\mathcal{I}_{-1}(\mu_nt)f(t)dt.
\end{equation}
\end{theorem}

Perhaps all versions of the Vorono\"{\dotlessi} summation formula before 1934 required that the arithmetic functions $a(n)$ and $b(n)$ satisfy \eqref{mainbessel} for $q=0$ or for an analogous identity when the arithmetic function, e.g., $d(n)$, the number of positive divisors of $n$, is generated by a Dirichlet series satisfying one of the alternative functional equations briefly alluded to above. However, Koshliakov \cite{kosh1}, \cite{kosh2} established a version of \eqref{voronoi} that sheds this hypothesis.  The version of Koshliakov's theorem that we offer below is more general than that established by Koshliakov and is in a different, but equivalent, formulation.  However, the proof of this more general result would follow along exactly the same lines as that of Koshliakov.  On the other hand, Koshliakov's theorem and our modest generalization below are less general than traditional versions of the Vorono\"{\dotlessi} summation formula because we require that $f$ be analytic.

\begin{theorem}\label{vv} \cite[p.~10]{kosh1} Let $0<\alpha<\beta$, and let $M$ and $N$ be integers such that $\lambda_{M-1}<\alpha<\lambda_{M}$ and $\lambda_{N}<\beta<\lambda_{N+1}$.  Suppose that $f(z)$ is an analytic function containing the interval $[\alpha,\beta]$ in its domain of analyticity. If  the infinite series and the integrals on the right side of \eqref{voronoi} converge uniformly on $[\alpha,\beta]$, then
\begin{equation}\label{voronoia}
{\sum_{\lambda_{M}\leq n\leq\lambda_{N}}}^{\prime}a(n)f(\lambda_n)=\int_{\alpha}^{\beta}{Q}^{\prime}_0(t)f(t)dt +\sum_{n=1}^{\infty}\df{b(n)}{\mu_n^{r-1}}
\int_{\alpha}^{\beta}\mathcal{I}_{-1}(\mu_nt)f(t)dt.
\end{equation}
\end{theorem}

In his hypotheses, Koshliakov assumes that the order of summation and integration on the right-hand side of \eqref{voronoia} can be
inverted.  Our hypothesis on the uniform convergence ensures that such an inversion of limiting operations is justified.  If we let  $\alpha\to0$ in Theorem \ref{vv}, we would obtain \eqref{voronoib} with $x$ replaced by $\beta$.

\section{Application of Theorems \ref{v} and \ref{vv}: Popov's Identity}\label{sect2}  We apply Theorems \ref{v} and \ref{vv} to
$$ \varphi(s)=\psi(s)=\zeta_k(s)=\sum_{n=1}^{\infty}r_k(n)n^{-s}, \quad \sigma > \tf12k,$$
where $r_k(n)$ is the number of representations of $n$ as a sum of $k$ squares.  As already indicated in \eqref{funcequa1}, the functional equation for $\zeta_k(s)$ is given by
$$ \pi^{-s}\Gamma(s)\varphi(s)=\pi^{-(k/2-s)}\Gamma(\tf12k-s)\varphi(\tf12k-s).$$
Thus,
$$ \lambda_n=\mu_n=\pi n,\quad n\geq1,\quad a(n)=b(n)=r_k(n),\quad n\geq1,\quad r=\tf12k.$$
Also,
\begin{equation}\label{Qa}
Q_0(x)=-1+\df{x^{k/2}}{\Gamma(\tf12k+1)}.
\end{equation}
For the properties above, one may  consult the paper by K.~Chandrasekharan and R.~Narasimhan \cite[pp.~18--19]{cn}.
We note that $\varphi(0) =-1$.  Thus, for suitable functions $f(x)$, by Theorems \ref{v} and \ref{vv},
\begin{align}\label{1}
{\sum_{\lambda_n\leq x}}^{\prime}r_k(n)f(\lambda_n)&=-\lim_{a\to0}f(a)+\int_0^x\df{t^{k/2-1}}{\Gamma(\tf12k)}f(t)dt\notag\\
&\quad+\sum_{n=1}^{\infty}\df{r_k(n)}{(\pi n)^{k/4-1/2}}\int_0^xt^{k/4-1/2}J_{k/2-1}(2\sqrt{\pi nt})f(t)dt.
\end{align}
Replacing $x$ by $\pi x$ in \eqref{1}, we find that
\begin{align}\label{2}
{\sum_{n\leq x}}^{\prime}r_k(n)f(\pi n)&=-\lim_{a\to0}f(a)+\int_0^{\pi x}\df{t^{k/2-1}}{\Gamma(\tf12k)}f(t)dt\notag\\
&\quad+\sum_{n=1}^{\infty}\df{r_k(n)}{(\pi n)^{k/4-1/2}}\int_0^{\pi x}t^{k/4-1/2}J_{k/2-1}(2\sqrt{\pi nt})f(t)dt.
\end{align}
In each of the two integrals above, set $t=\pi u$.  Hence, \eqref{2} becomes
\begin{align}\label{3}
{\sum_{n\leq x}}^{\prime}r_k(n)f(\pi n)&=-\lim_{a\to0}f(a)+\df{\pi^{k/2}}{\Gamma(\tf12k)}\int_0^{ x}u^{k/2-1}f(\pi u)du\notag\\
&\quad+\pi\sum_{n=1}^{\infty}\df{r_k(n)}{ n^{k/4-1/2}}\int_0^{ x}u^{k/4-1/2}J_{k/2-1}(2\pi\sqrt{nu})f(\pi u)du.
\end{align}
Now, replace $f(\pi x)$ by
$$ \df{f(x)}{(\pi x)^{k/4-1/2}}$$
and then multiply both sides of the equation by $\pi^{k/4-1/2}$ to arrive at
\begin{align}\label{4}
{\sum_{n\leq x}}^{\prime}r_k(n)\df{f(n)}{n^{k/4-1/2}}&=-\lim_{a\to0}\df{f(a)}{a^{k/4-1/2}}+\df{\pi^{k/2}}{\Gamma(\tf12k)}\int_0^{ x}u^{k/4-1/2}f(u)du\notag\\
&\quad+\pi\sum_{n=1}^{\infty}\df{r_k(n)}{ n^{k/4-1/2}}\int_0^{ x}J_{k/2-1}(2\pi\sqrt{nu})f( u)du.
\end{align}
The identity \eqref{4} is the identity at the bottom of the first page of Popov's paper \cite{popov}.

We now set
$$ f(u):=J_{k/2-1}(2\pi\sqrt{zu})e^{-\pi tu}, \quad z>0,\quad \text{Re }t>0,$$
and let $x\to\infty$ in \eqref{4}.

First,  we need to calculate
\begin{equation}\label{5}
\lim_{a\to0}\df{J_{k/2-1}(2\pi\sqrt{za})e^{-\pi ta}}{a^{k/4-1/2}}=\df{\pi^{k/2-1}z^{k/4-1/2}}{\Gamma(\tf12k)},
\end{equation}
where we have used the definition of the Bessel function of order $\nu$ \cite[p.~40, eq.~(8)]{watson}.

Second, setting $u=x^2$, we evaluate the integral
\begin{align}\label{6}
\int_0^{ \infty}u^{k/4-1/2}J_{k/2-1}(2\pi\sqrt{zu})e^{-\pi tu}du&=2\int_0^{\infty}x^{k/2}J_{k/2-1}(2\pi x\sqrt{z})e^{-\pi tx^2}dx\notag\\
&=\df{z^{k/4-1/2}}{\pi t^{k/2}}e^{-\pi z/t},
\end{align}
where we have used the evaluation \cite[eq.~6.631, no.~4.]{gr}.

Third, setting $u=x^2$, we find that
\begin{align}\label{7}
\int_0^{ \infty}J_{k/2-1}(2\pi\sqrt{nu})f( u)du&=\int_0^{ \infty}J_{k/2-1}(2\pi\sqrt{nu})J_{k/2-1}(2\pi\sqrt{zu})e^{-\pi tu}du\notag\\
&=2\int_0^{ \infty}x J_{k/2-1}(2\pi x\sqrt{n})J_{k/2-1}(2\pi x\sqrt{z})e^{-\pi tx^2}dx\notag\\
&=\df{1}{\pi
  t}\exp\left(-\df{\pi(n+z)}{t}\right)I_{k/2-1}\left(\df{2\pi\sqrt{nz}}{t}\right),
\end{align}
where $I_{\nu}(z)$ is the Bessel function of imaginary argument of order $\nu$, and where we have used an evaluation found in \cite[eq.~6.633, no.~2]{gr}.

If we now substitute \eqref{5}--\eqref{7} into \eqref{4}, we conclude the proof of Popov's beautiful theorem below, which is precisely equation (6) in Popov's paper \cite{popov} and \eqref{88} above.
\begin{theorem}\label{popovtheorem} For $\text{Re } t>0$,
\begin{align}\label{8}
&\df{\pi^{k/2-1}z^{k/4-1/2}}{\Gamma(\tf12k)}+\sum_{n=1}^{\infty}r_{k}(n)\df{J_{k/2-1}(2\pi\sqrt{nz})}{n^{k/4-1/2}}e^{-\pi nt}\notag\\
&=\df{e^{-\pi z/t}}{t}\left\{\df{\pi^{k/2-1}z^{k/4-1/2}}{t^{k/2-1}\Gamma(\tf12k)}
+\sum_{n=1}^{\infty}r_{k}(n)\df{I_{k/2-1}\left(\df{2\pi\sqrt{nz}}{t}\right)}{
  n^{k/4-1/2}}e^{-\pi n/t}\right\}.
\end{align}
\end{theorem}

  It is clear that we can apply Theorems \ref{v} and \ref{vv} to derive similar identities for arithmetical functions generated by a Dirichlet series satisfying a functional equation of the type  \eqref{funcequa}.  Thus, for example, we can state a similar identity when $a(n)=\tau(n)$, Ramanujan's tau-function. We can also replace the quadratic form $m^2+n^2$ by any positive definite quadratic form $Q(x,y)$.  Thus, we can derive a similar result for $r(Q,n)$, the number of representations of the positive integer $n$ by the quadratic form $Q$.  In Section \ref{sun}, we provide further application of Theorems \ref{v} and \ref{vv}.

\section{An Analogue of a Theorem of Ramanujan from His Lost Notebook}\label{sun}

Recall that in our remark at the end of Section \ref{sect2}, we mentioned that Theorems \ref{v} and \ref{vv} can be applied to other arithmetical functions generated by Dirichlet series satisfying a functional equation with a simple gamma factor. In this section, we establish such an analogue.  Then, we use this analogue to derive a formula which is in the same spirit as a remarkable identity from Ramanujan's lost notebook \cite{lnb}, which was first established by three of the present authors  in \cite{bsa1}.


For a Dirichlet character $\chi,$ we define
$$d_{\chi}(n)=\sum_{d \mid n}\chi(d).$$

\begin{theorem}
Let $\chi$ be an odd primitive character modulo $q,$ and let $z>0$ and $\text{Re}~ t>0.$ Then,
\begin{align}\label{odd}
\sum_{n=1}^{\infty}d_{\chi}(n)J_0(4\pi\sqrt{nz})e^{-n\pi t}&=-\dfrac{1}{2}L(0, \chi)+\frac{L(1,\chi)}{\pi t}e^{-4\pi z/t} \notag\\
&\quad -\dfrac{2i\tau(\chi)}{tq}
\sum_{n=1}^{\infty}d_{\overline{\chi}}(n)e^{-4\pi(n/q+z)/t}I_0\Big(\frac{8\pi}{t}\sqrt{nz/q}\Big),
\end{align}
where $\tau(\chi)$ denotes the Gauss sum
\begin{align*}
\tau(\chi):=\sum_{h=1}^{q-1}\chi(h)e^{2\pi ih/q}.
\end{align*}
\end{theorem}

\begin{proof} If $\chi$ is an odd primitive character modulo $q,$ then from the functional equations of the Riemann zeta function and the Dirichlet $L$-function associated with $\chi$ \cite[pp.~59, 71]{davenport},
\begin{align}\label{ofe}
\Big(\frac{2\pi}{\sqrt{q}}\Big)^{-s}\Gamma(s)L(s, \chi)\zeta(s)
=-\frac{i\tau(\chi)}{\sqrt{q}}\Big(\frac{2\pi}{\sqrt{q}}\Big)^{s-1}\Gamma(1-s)L(1-s,\overline{\chi})\zeta(1-s).
\end{align}
We let
$$\varphi(s):=\Big(\frac{2\pi}{\sqrt{q}}\Big)^{-s}L(s,\chi)\zeta(s)=\Big(\frac{2\pi}{\sqrt{q}}\Big)^{-s}
\sum_{n=1}^{\infty}\frac{d_{\chi}(n)}{n^s}$$
and
$$\psi(s):=-\frac{i\tau(\chi)}{\sqrt{q}}\Big(\frac{2\pi}{\sqrt{q}}\Big)^{-s}L(s,\overline{\chi})\zeta(s)
=-\frac{i\tau(\chi)}{\sqrt{q}}\Big(\frac{2\pi}{\sqrt{q}}\Big)^{-s} \sum_{n=1}^{\infty}\frac{d_{\overline{\chi}}(n)}{n^s}.$$
Then, from \eqref{ofe}, it follows that
$$\Gamma(s)\varphi(s)=\Gamma(1-s)\psi(1-s).$$
Also, by Theorem 12 in \cite{bsa1}, we see that $A_0(x)=Q_0(x)+D_0(x),$
where $a(n)=d_{\chi}(n),$ $\lambda_n=\mu_n=2\pi n/\sqrt{q},$ and  $b(n)=-i\tau(\chi)d_{\overline{\chi}}(n)/\sqrt{q}.$
Thus, using Theorems \ref{v} or \ref{vv}, and \cite[Theorem 12]{bsa1}, we deduce that
\begin{align}\label{oddthm}
{\sum_{\lambda_n\leq x}}^{\prime}d_{\chi}(n)f(\lambda_n)&=-\frac12L(0, \chi)\lim_{a\rightarrow 0}f(a)
+\int_0^x Q'_0(t)f(t) \, dt \notag\\&\quad-\frac{i\tau(\chi)}{\sqrt{q}}\sum_{n=1}^{\infty}d_{\overline{\chi}}(n)\int_0^x J_0(2\sqrt{\mu_n t}) f(t) \, dt.
\end{align}
Since $\zeta(0)=-\tf12$ and $\zeta(s)$ has a simple pole at $s=1$ with residue $1$, we find that
$$Q_0(t)=\frac{1}{2\pi i}\int_{C_{0}} \frac{\varphi(s)}{s}t^s \, ds=-\df12 L(0, \chi)+\frac{\sqrt{q}}{2\pi}L(1,\chi)t,$$
and so $$Q'_0(t)=\frac{\sqrt{q}}{2\pi}L(1,\chi).$$
Replacing $x$ by $2\pi x/\sqrt{q}$ in \eqref{oddthm}, we thus have
\begin{align}\label{oddthm1}
{\sum_{n\leq x}}^{\prime}d_{\chi}(n)f(2\pi n/\sqrt{q})&=-\frac12L(0, \chi)\lim_{a\rightarrow 0}f(a)
+\frac{\sqrt{q}}{2\pi}L(1, \chi)\int_0^{2\pi x/\sqrt{q}} f(t) \, dt \notag\\
&\quad-\frac{i\tau(\chi)}{\sqrt{q}}\sum_{n=1}^{\infty}d_{\overline{\chi}}(n)\int_0^{2\pi x/\sqrt{q}}  J_0(2\sqrt{2\pi n t}/q^{1/4}) f(t) \, dt.
\end{align}
In each of the integrals in \eqref{oddthm1}, we set $t=2\pi u/\sqrt{q}.$  It follows that
\begin{align}\label{oddthm2}
{\sum_{n\leq x}}^{\prime}d_{\chi}(n)f(2\pi n/\sqrt{q})&=-\frac12L(0, \chi)\lim_{a\rightarrow 0}f(a)
+L(1, \chi)\int_0^{x} f(2\pi u/\sqrt{q}) \, du \notag\\
&\quad -\frac{2\pi i\tau(\chi)}{q}\sum_{n=1}^{\infty}d_{\overline{\chi}}(n)\int_0^{x}  J_0(4\pi\sqrt{nu/q}) f(2\pi u/\sqrt{q}) \, du.
\end{align}
Also,  replace $f(2\pi x/\sqrt{q})$ by $$g(x):=J_0(4\pi\sqrt{zx/q})e^{-\pi tx/q},$$ and let $x\rightarrow \infty$ in \eqref{oddthm2} to obtain
\begin{align}\label{oddthm3}
\sum_{n=1}^{\infty}d_{\chi}(n)J_0(4\pi\sqrt{zn/q})e^{-\pi tn/q}&=-\frac12L(0, \chi)g(0)
+L(1, \chi)\int_0^{\infty} g(u) \, du \\
&\quad -\frac{2\pi i\tau(\chi)}{q}\sum_{n=1}^{\infty}d_{\overline{\chi}}(n)\int_0^{\infty}  J_0(4\pi\sqrt{nu/q}) g(u)\, du.\notag
\end{align}
We note that $g(0)=1.$ We now evaluate the integrals in \eqref{oddthm3}.
First observe that
\begin{align}\label{g}
\int_0^{\infty} g(u) \, du&=\int_0^{\infty} J_0(4\pi\sqrt{zu/q})e^{-\pi tu/q} \, du \notag\\
&=2\int_0^{\infty} xJ_0(4\pi x \sqrt{z/q})e^{-\pi tx^2/q} \, dx  \notag\\
&=\frac{q}{\pi t}e^{-4\pi z/t},
\end{align}
where we set $u=x^2$ for the second equality, and we used \cite[eq. 6.631, no. 4]{gr} for the last equality.

Similarly, we find that
\begin{align}\label{Jg}
\int_0^{\infty}  J_0(4\pi\sqrt{nu/q}) g(u)\, du&=\int_0^{\infty}  J_0(4\pi\sqrt{nu/q})J_0(4\pi\sqrt{zu/q})e^{-\pi tu/q}\, du \notag\\
&=2\int_0^{\infty} xJ_0(4\pi x\sqrt{n/q})J_0(4\pi x \sqrt{z/q})e^{-\pi tx^2/q} \, dx  \notag\\
&=\frac{q}{\pi t}e^{-4\pi(n+z)/t}I_0(8\pi\sqrt{nz}/t),
\end{align}
where we used \cite[eq. 6.633, no. 2]{gr}.
Putting \eqref{g} and \eqref{Jg} into \eqref{oddthm3}, and replacing $z$ and $t$ by $zq$ and $tq,$ respectively, we complete the proof of \eqref{odd}.
\end{proof}

Next, we derive the following theorem from \eqref{odd}.

\begin{theorem}\label{pop} Let $0<\theta<1, z>0$ and $\text{Re}~ t>0.$ Then,
\begin{align}\label{sine}
&\sum_{n=1}^{\infty}J_0\big(4\pi\sqrt{nz}\big)e^{-n\pi t}\sum_{d \mid n}\sin(2\pi d\theta)=-\frac14\cot(\pi\theta)+\frac{e^{-4\pi z/t}}{t}\Big(\frac12-\theta\Big) \notag\\
& \qquad+\frac{e^{-4\pi z/t}}{t}\left\{\sum_{m=1}^{\infty}\sum_{r=0}^{\infty}\frac{I_0\Big(8\pi\sqrt{mz(r+\theta)}/t\Big)}{e^{4\pi m(r+\theta)/t}}
-\frac{I_0\Big(8\pi\sqrt{mz(r+1-\theta)}/t\Big)}{e^{4\pi m(r+1-\theta)/t}} \right\}.
\end{align}
\end{theorem}

\begin{proof}
Note that from \cite[p.~240, eq. (9.54)]{temme}, \cite[p.~203]{watson}, for Re $w>0$ and for large values of $|w|$,
\begin{equation}\label{iasym}
I_{\nu}(w)\sim\frac{e^w}{\sqrt{2\pi w}}\sum_{m=0}^{\infty}(-1)^m\frac{(\nu,m)}{(2w)^m},
\end{equation}
where 
\begin{align*}
(\nu,m)=\frac{\Gamma(\nu+m+1/2)}{\Gamma(m+1)\Gamma(\nu-m+1/2)}.
\end{align*}
From \eqref{iasym}, we observe that the double series on the right-hand side of \eqref{sine} converges absolutely and uniformly on any compact interval for $\theta \in (0,1)$
and so represents a continuous function of $\theta$ there.
Therefore, it suffices to prove \eqref{sine} for $\theta=a/q,$ where $q$ is prime and  $0<a<q.$
We first multiply both sides of \eqref{odd} by $\chi(a)\tau(\overline{\chi})/i\phi(q),$ where $\phi(q)$ is Euler's totient function, and sum on odd primitive characters $\chi$ modulo $q.$
Then, the left-hand side of \eqref{odd} becomes
\begin{align}\label{sine1}
&\frac{1}{i\phi(q)}\sum_{\chi ~ \text{odd}}\chi(a)\tau(\overline{\chi})\sum_{n=1}^{\infty}d_{\chi}(n)J_0(4\pi\sqrt{nz})e^{-n\pi t} \notag\\
&\qquad =\frac{1}{i\phi(q)}\sum_{n=1}^{\infty}J_0(4\pi\sqrt{nz})e^{-n\pi t}\sum_{\chi ~ \text{odd}}\chi(a)\tau(\overline{\chi})\sum_{d \mid n}\chi(d) \notag\\
&\qquad=\sum_{n=1}^{\infty}J_0(4\pi\sqrt{nz})e^{-n\pi t}\sum_{d \mid n}\frac{1}{i\phi(q)}\sum_{\chi ~ \text{odd}}\chi(a)\tau(\overline{\chi})\chi(d) \notag\\
&\qquad=\sum_{n=1}^{\infty}J_0(4\pi\sqrt{nz})e^{-n\pi t}\sum_{d \mid n}\sin\Big(\frac{2\pi da}{q}\Big),
\end{align}
where the last equality follows from the formula \cite[Lemma 2.5]{bkz3}
\begin{align*}
\sin\Big(\frac{2\pi na}{q}\Big)=\frac{1}{i\phi(q)}\sum_{\chi ~ \text{odd}}\chi(a)\tau(\overline{\chi})\chi(n).
\end{align*}

Now, examine the right-hand side of \eqref{odd}.
In order to evaluate the first term on the right-hand side of \eqref{odd}, we recall from \cite[p.~2072]{bsa1} or \cite[p.~11]{davenport}, and
\cite[eq.~(2.8)]{bkz4}
\begin{align*}
L(0, \chi)=-\frac{i\tau(\chi)}{\pi}L(1, \overline{\chi})=\frac{i}{2\tau(\overline{\chi})}\sum_{1\leq h<q}\overline{\chi}(h)\cot\Big(\frac{\pi h}{q}\Big).
\end{align*}
We also need the formula \cite[eq.~(3.7)]{bsa1}
\begin{align}\label{oddah}
\sum_{\chi ~ \text{odd}}\chi(a)\overline{\chi}(h) =
\begin{cases}\pm \frac12\phi(q), & \text{if} \quad h\equiv \pm a \pmod q,\\
0, & \text{otherwise}.
\end{cases}
\end{align}
Then,
{\allowdisplaybreaks\begin{align}\label{sine2}
&\frac{1}{i\phi(q)}\sum_{\chi ~ \text{odd}}\chi(a)\tau(\overline{\chi})L(0, \chi)\notag \\
&\qquad =\frac{1}{2\phi(q)}\sum_{1\leq h<q}\cot\Big(\frac{\pi h}{q}\Big)\sum_{\chi ~ \text{odd}}\chi(a)\overline{\chi}(h) \notag\\
&\qquad =\frac14\left\{\cot\Big(\frac{\pi a}{q}\Big)-\cot\Big(\frac{\pi(q-a)}{q}\Big)\right\} \notag\\
&\qquad =\frac12\cot\Big(\frac{\pi a}{q}\Big).
\end{align}}
Thus, we obtain the first term on the right-hand side of \eqref{sine}.

Next, using the formula \cite[eq.~(3.5)]{bsa1}
\begin{align*}
\tau(\overline{\chi})L(1,\chi)=\pi i\sum_{1\leq h<q}\overline{\chi}(h)\Big(\frac12-\frac{h}{q}\Big),
\end{align*}
we obtain
\begin{align}\label{sine3}
&\frac{1}{i\phi(q)}\sum_{\chi ~ \text{odd}}\chi(a)\tau(\overline{\chi})L(1,\chi) \notag\\
&\qquad =\frac{\pi}{\phi(q)}\sum_{\chi ~ \text{odd}}\chi(a)\sum_{1\leq h<q}\overline{\chi}(h)\Big(\frac12-\frac{h}{q}\Big) \notag\\
&\qquad=\frac{\pi}{\phi(q)}\sum_{1\leq h<q}\Big(\frac12-\frac{h}{q}\Big)\sum_{\chi ~ \text{odd}}\chi(a)\overline{\chi}(h) \notag\\
&\qquad=\pi\Big(\frac12-\frac{a}{q}\Big),
\end{align}
where the last equality follows from \eqref{oddah}.  This gives the second expression on the right-hand side of \eqref{sine}.

Lastly, using the fact that $\tau(\chi)\tau(\overline{\chi})=-q$ and employing \eqref{oddah} once again, we deduce that
{\allowdisplaybreaks\begin{align}\label{sine4}
&-\frac{1}{i\phi(q)}\sum_{\chi ~ \text{odd}}\chi(a)\tau(\overline{\chi})\frac{2i\tau(\chi)}{q}
\sum_{n=1}^{\infty}d_{\overline{\chi}}(n)e^{-4\pi n/(tq)}I_0\Big(\frac{8\pi}{t}\sqrt{nz/q}\Big) \notag\\
&\qquad=\frac{2}{\phi(q)}\sum_{n=1}^{\infty}e^{-4\pi n/(tq)}I_0\Big(\frac{8\pi}{t}\sqrt{nz/q}\Big)
\sum_{\chi ~ \text{odd}}\chi(a)\sum_{d\mid n}\overline{\chi}(d) \notag\\
&\qquad=\frac{2}{\phi(q)}\sum_{n=1}^{\infty}e^{-4\pi n/(tq)}I_0\Big(\frac{8\pi}{t}\sqrt{nz/q}\Big)
\sum_{d\mid n}\sum_{\chi ~ \text{odd}}\chi(a)\overline{\chi}(d) \notag\\
&\qquad=\sum_{n=1}^{\infty}e^{-4\pi n/(tq)}I_0\Big(\frac{8\pi}{t}\sqrt{nz/q}\Big)
\Big(\sum_{\substack{d\mid n\\d\equiv a \pmod q}}1- \sum_{\substack{d\mid n\\d\equiv -a \pmod q}}1\Big) \notag\\
&\qquad=\sum_{m=1}^{\infty}\sum_{r=0}^{\infty}\left\{\frac{I_0\Big(\frac{8\pi}{t}\sqrt{mz(r+a/q)}\Big)}{e^{4\pi m(r+a/q)/t}}
 -\frac{I_0\Big(\frac{8\pi}{t}\sqrt{mz(r+1-a/q)}\Big)}{e^{4\pi m(r+1-a/q)/t}}\right\},
\end{align}}%
where we replace $n$ by $m(rq+a)$ and $m(rq+q-a)$, respectively, in the last line.
Hence, using \eqref{odd} and  \eqref{sine2}--\eqref{sine4}, we complete the proof.
\end{proof}

\begin{center}
\textbf{Acknowledgements}
\end{center}
We sincerely thank Dmitry Vasilenko, Vice-Rector for International Relations at St.~Petersburg State University of Economics, for providing biographical information of A.~I.~Popov as well as a list of his $12$ publications. We also thank Professor Alexander Kurganov of the Department of Mathematics, Tulane University, and Professor Olga S.~Rozanova of the Faculty of Mechanics and Mathematics, Moscow State University, for sending us a copy of the paper \cite{kosh2}.

\end{document}